\documentclass{amsart}

\usepackage{amsmath}
\usepackage{amssymb}

\begin{document}

\title{A Matrix form of Ramanujan-type series for $1/\pi$}

\author{Jes\'{u}s Guillera}

\address{Av. Cesáreo Alierta, 31, esc. iz. $4^{o}$--A, Zaragoza, Spain.}

\email{jguillera@gmail.com}



\maketitle

\newtheorem{theorem}{Theorem}[section]
\newtheorem{corollary}[theorem]{Corollary}
\newtheorem{lemma}[theorem]{Lemma}
\newtheorem{proposition}[theorem]{Proposition}
\newtheorem{definition}[theorem]{Definition}
\newtheorem{example}[theorem]{Example}
\newtheorem{remark}[theorem]{Remark}
\newtheorem{notation}[theorem]{Notation}
\newtheorem{question}[theorem]{Question}
\newtheorem{claim}[theorem]{Claim}
\newtheorem{conjecture}[theorem]{Conjecture}
\newtheorem{expansion}[theorem]{Expansion}

\numberwithin{equation}{section}

\renewcommand{\baselinestretch}{1}

\begin{abstract}
In this paper we prove theorems related to the Ramanujan-type series for $1/\pi$ (type $_3F_2$) and to the Ramanujan-like series, discovered by the author, for $1/\pi^2$ (type $_5F_4$). Our developments for the cases $_3 F_2$ and $_5 F_4$ connect with the theory of modular functions and with the theory of Calabi-Yau differential equations, respectively.
\end{abstract}

\section{Introduction}
In $1914$ Ramanujan discovered four families of series for $1/\pi$, corresponding to the values of $s=1/2$, $s=1/4$, $s=1/3$ and $s=1/6$, which are of the form:
\begin{equation}\label{rama-series}
\sum_{n=0}^{\infty} z^n \frac{ \left( \frac{1}{2} \right)_n (s)_n (1-s)_n}{ (1)_n^3} (a+bn)= \frac{1}{\pi},
\end{equation}
where $-1 \leq z<1$, $a$ and $b$ are algebraic numbers \cite{ramanujan}. The symbol $(s)_n$ used in the series is the rising factorial or Pochhammer symbol which is defined by
\begin{equation} \label{poch}
(s)_n=s(s+1)\cdots(s+n-1)
\end{equation}
for $n=1,2,\dots,$ and equal to $1$ for $n=0$. He gave $17$ examples of series of such type. About seventy years later Ramanujan's work related to these series began to be understood, and since then many other series of this type have been found and proved by using the theory of elliptic modular functions. For example, J. and P. Borwein, in their book \cite{borweinagm}, prove the $17$ series found by Ramanujan. In \cite{chudnovsky} D. and G. Chudnovsky obtained the fastest convergent Ramanujan-type series with a rational value of $z$. In \cite{chan} their authors derive some new series with $s=1/3$, somewhat implicit in Ramanujan's work. Very recently B. Berndt and N. Baruah in \cite{baruah1} and \cite{baruah2} get to follow more closely the initial ideas of Ramanujan as presented in Section 13 of the celebrated paper \cite{ramanujan}, and prove many new series of this type. In $2002$ we discovered similar families of series $1/\pi^2$, which are of the form
\begin{equation}\label{rama-gui}
\sum_{n=0}^{\infty} z^n \frac{ \left( \frac{1}{2} \right)_n (s)_n (t)_n (1-t)_n(1-s)_n}{ (1)_n^5} (a+bn+cn^2)= \frac{1}{\pi^2},
\end{equation}
where $-1 \leq z<1$, $a$, $b$ and $c$ are algebraic numbers. In the latter situation, there are $14$ possible couples of values for $(s,t)$, namely: $(1/2,1/2)$, $(1/2,1/3)$, $(1/2,1/4)$, \; $(1/2,1/6)$, \: $(1/3,1/3)$, \: $(1/3,1/4)$, \: $(1/3,1/6) $, \: $(1/4,1/4)$, \: $(1/4,1/6)$, \: $(1/6,1/6)$, $(1/5,2/5)$, $(1/8,3/8)$, $(1/10,3/10)$ and $(1/12,5/12)$ (see \cite{guilleraEMpi2}). Later, in \cite{guilleraEMconj}, we proposed some conjectures, for the cases $1/\pi$ and $1/\pi^2$. They motivate the study, that we are going to make, of the following expansions as $x \to 0$:

\begin{expansion}\label{conje1}
\begin{equation}\nonumber
|z|^x \sum_{n=0}^{\infty} z^n  \frac{ \left( \frac{1}{2} \right)_{n+x} (s)_{n+x} (1-s)_{n+x}}{ (1)_{n+x}^3} (a+b(n+x))=\frac{1}{\pi}-\frac{k \pi}{2}x^2+O(x^3).
\end{equation}
\end{expansion}

\begin{expansion}\label{conje2}
\begin{multline}\nonumber
|z|^x \sum_{n=0}^{\infty} z^n  \frac{ \left( \frac{1}{2} \right)_{n+x} (s)_{n+x} (t)_{n+x} (1-t)_{n+x} (1-s)_{n+x}}{ (1)_{n+x}^5} (a+b(n+x)+c(n+x)^2)=\\  \frac{1}{\pi^2}-\frac{k}{2}x^2+\frac{j}{24} \pi^2 x^4+O(x^5).
\end{multline}
\end{expansion}
In these expansions we consider the $z$-analytic objects in the half-plane $u\cdot\Re(z)>0$, where $u$ is the sign of $\Re(z)$; in particular, we replace $|z|$ with $uz$ in the discussion below. We also need the generalized definition of the Pochhammer symbol, namely: $(s)_x=\Gamma(s+x)/\Gamma(s)$, which reduces to (\ref{poch}) when $x$ is a non-negative integer. Note that Expansion \ref{conje2} was stated in \cite{guilleraEMconj} in a weaker form.
\\ \par Our idea for this work consists in replacing the variable $x$ with a fix nilpotent matrix $X$ of order three and five, respectively. This is a natural way of truncating the corresponding Taylor series and Lemma \ref{lema} implies that we do not lose the information required. In addition, this leads to a simpler formulation, since we get rid of the derivatives with respect to $x$.
\\ \par All matrices are the result of substitution of a nilpotent matrix $X$ in analytic functions $f(x)$. In other words, they are of the form
\[ A=f(X)=a_0 I+a_1 X + \cdots + a_{n-1} X^{n-1}, \]
where $n$ is the order of $X$ and $I$ the identity matrix. We say that $a_0$, $a_1$, \dots, $a_{n-1}$ are the components (or coefficients) of $A$. The product of such matrices is commutative; if $a_0\neq0$ the matrix $A$ has an inverse $A^{-1}$. For convenience, we also use the notation $A^{-1}=I/A$.
\\ \par In our computations with Maple 9, we choose nilpotent matrices $X$ whose unique nonzero entries are equal to $1$ in the positions $(j, j+1)$ of the matrix. With this choice, the components of a matrix $A$ are simply the entries of its first row.
\begin{lemma}\label{lema}
If $f(x)$ is an analytic function and
\begin{equation}\label{fXc}
f(X)=c_0 I+c_1 X + \cdots + c_{n-1} X^{n-1},
\end{equation}
where $X$ is a nilpotent matrix of order $n$, then
\[ f(x)=c_0 + c_1 x + \cdots + c_{n-1} x^{n-1}+O(x^n). \]
\end{lemma}
\begin{proof}
Since $f(x)$ is an analytic function, it admits a power expansion,
\[ f(x)=d_0 + d_1 x + \cdots + d_{n-1} x^{n-1}+O(x^n). \]
Substituting $X$ for $x$, we obtain
\begin{equation}\label{fXd}
f(X)=d_0 I+d_1 X + \cdots + d_{n-1} X^{n-1}.
\end{equation}
The subtraction of (\ref{fXc}) from (\ref{fXd}) gives
\[ (d_0-c_0) I+(d_1-c_1) X + \cdots + (d_{n-1}-c_{n-1}) X^{n-1}=0. \]
Multiplying by $X^{n-1}$, we obtain $d_0=c_0$, hence
\[ (d_1-c_1) X + \cdots + (d_{n-1}-c_{n-1}) X^{n-1}=0. \]
Further, multiplying by $X^{n-2}$ we obtain $d_1=c_1$, and so on.
\end{proof}
We define the gamma function of a non-degenerated matrix $A$ as follows:
\[ \Gamma(A)=\int_{0}^{\infty} t^{A-I} e^{-t} dt. \]
For example, if $X$ is a nilpotent matrix of order $2$, we have
\[ \Gamma(I+X)=\int_{0}^{\infty} t^{X} e^{-t} dt=\int_{0}^{\infty} (I+X \ln t) e^{-t} dt=I-\gamma X.  \]
For $n=1, 2, \dots$, we define the Pochhammer symbol of a matrix $A$ as
\[ (A)_n=\prod_{j=0}^{n-1} (A+jI)=\frac{\Gamma(A+nI)}{\Gamma(A)}, \]
and we generalize this notion to matrix-valued indices: $(A)_B=\Gamma(A+B)/\Gamma(A)$. For abuse of notation, we often identify a number $n$ with the corresponding matrix $nI$. This has been done, for example, in the last definition.

\section{Matrix form of Expansion \ref{conje1}}

We obtain an equivalent of Expansion \ref{conje1} in a matrix form if we replace $x$ by a nilpotent matrix $X$ of order $3$. This equivalence is a consequence of Lemma \ref{lema}. In the notation
\[ P_n(s,X)= \frac{ \left( \frac{1}{2}I+X \right)_n (sI+X)_n ((1-s)I+X)_n}{ (I+X)_n^3}, \]
and with the property $(s)_{n+x}=(s+x)_n(s)_x$ we have
\begin{expansion}[Expansion \ref{conje1} in matrix form]\label{matrix3F2}
\begin{equation}\nonumber
\sum_{n=0}^{\infty} z^n P_n(s,X) \left( \frac{}{} \! aI+b(nI+X) \right)=
(uz)^{-X} P_X^{-1}(s,0) \left(\frac{1}{\pi}-\frac{k \pi}{2}X^2 \right),
\end{equation}
where $X$ is an arbitrary nilpotent matrix of order $3$.
\end{expansion}
If we denote
\[
A=\sum_{n=0}^{\infty} z^n P_n(s,X), \qquad B=\sum_{n=0}^{\infty} z^n P_n(s,X)(nI+X),
\]
and
\[
M=(uz)^{-X} P_X^{-1}(s,0) \left( \frac{1}{\pi} I -\frac{k \pi}{2} X^2 \right),
\]
then Expansion \ref{matrix3F2} can be written in the form $aA+bB=M$, which is equivalent to the system
\begin{equation}\label{sistema1}
\left(
  \begin{array}{ccc}
    a_{0} & b_{0} \\
    a_{1} & b_{1} \\
    a_{2} & b_{2} \\
  \end{array}
\right)
\left(
  \begin{array}{c}
    a \\
    b \\
  \end{array}
\right)=
\left(
  \begin{array}{c}
    m_{0} \\
    m_{1} \\
    m_{2} \\
  \end{array}
\right).
\end{equation}

\subsection{Picard-Fuchs differential equation}\label{sec-picard}
In Proposition \ref{propo1} we will see that the components of the matrix $Y=(uz)^X A$, namely
\begin{equation}\label{y0y1y2}
y_0=a_0, \qquad y_1=a_0 \ln(uz)+a_1, \qquad y_2=a_0 \frac{\ln^2(uz)}{2}+a_1 \ln(uz)+a_2,
\end{equation}
satisfy a Picard-Fuchs differential equation by proving that the matrix $Y$ itself satisfies that equation. In Proposition \ref{propo-picard-1-2} we will see that these solutions are the squares of the solutions of a simpler Picard-Fuchs differential equation.

\begin{proposition}\label{propo1}
The matrix $Y=(uz)^X A$ is a fundamental solution of the differential equation
\begin{equation}\label{picard-fu}
\left( z\frac{d}{dz} \right)^3 Y - z \left( z\frac{d}{dz}+\frac{1}{2} \right) \left( z\frac{d}{dz}+s \right) \left( z\frac{d}{dz}+1-s \right) Y=0.
\end{equation}
\end{proposition}
\begin{proof}
We give details for the case $s=1/2$, but the other cases are similar. Writing
\[ A=\sum_{n=0}^{\infty} A_n z^n, \quad {\rm where} \quad A_n=\frac{\left( \frac{1}{2} I + X \right)_n^3}{( I + X )_n^3}, \]
we have the recurrence
\begin{equation}\label{recu1}
\left[ \frac{}{} (n+1)I+X \right]^3 A_{n+1}=\left[ \left(n+\frac{1}{2} \right)I+X \right]^3 A_n.
\end{equation}
If we substitute $Y=(uz)^X A$ in the left-hand side of (\ref{picard-fu}), we obtain
\[
(uz)^X \sum_{n=0}^{\infty} A_n
\left( \frac{}{} nI+X \right)^3 z^n-(uz)^X \sum_{n=0}^{\infty} A_n \left[ \left( n+\frac{1}{2} \right)I+X \right]^3 z^{n+1},
\]
and using (\ref{recu1}), we see that it is equal to $(uz)^X X^3,$ which is the zero matrix of order $3$.
\end{proof}
\begin{proposition}\label{propo-picard-1-2}{\rm (Clausen's identity \cite[p. 178]{borweinagm})}
If $y$ is a solution of the differential equation
\begin{equation}\label{picard1}
\left( z\frac{d}{dz} \right)^2 y - z \left( z\frac{d}{dz}+\frac{s}{2} \right)
\left( z\frac{d}{dz}+\frac{1-s}{2} \right) y=0,
\end{equation}
then $c y^2$, for an arbitrary constant $c$, is a solution of the differential equation
\begin{equation}\label{picard2}
\left( z\frac{d}{dz} \right)^3 y - z \left( z\frac{d}{dz}+\frac{1}{2} \right) \left( z\frac{d}{dz}+s \right) \left( z\frac{d}{dz}+1-s \right) y=0.
\end{equation}
\end{proposition}
\begin{proof}
Equation (\ref{picard1}) relates $y''$ to $y$ and $y'$, and therefore it also relates $y'''$ to $y$ and $y'$; just substitute these relations into (\ref{picard2}) to see that everything cancels.
\end{proof}

\subsection{The components of $A$ and $B$}\label{sec-ele}
We will find relations among the components of the matrices $A$ and $B$.
\begin{lemma}
The components of the matrix $B$ are related to those of $A$ in the way
\begin{equation}\label{b11b12b13}
b_{0}=za_{0}', \qquad b_{1}=a_{0}+za_{1}', \qquad b_{2}=a_{1}+za_{2}'.
\end{equation}
\end{lemma}
\begin{proof}
The matrix $B$ can be written as
\[ B=z A'+XA=za_{0}' I+ (a_{0}+za_{1}')X+(a_{1}+za_{2}')X^2, \]
whose components imply (\ref{b11b12b13}).
\end{proof}
\begin{lemma}
The components of the matrices $A$ and $B$ satisfy the following relations:
\begin{equation}\label{3F2relacion1}
a_{1}^2=2 a_{0} a_{2},
\end{equation}
\begin{equation}\label{3F2relacion2}
a_1b_1-a_0b_2-a_2b_0=0,
\end{equation}
\begin{equation}\label{3F2relacion3}
b_1^2-2b_0b_2=\frac{1}{1-z}.
\end{equation}
\end{lemma}
\begin{proof}
We first define the matrix
\[ \sum_{n=0}^{\infty} z^n \frac{\left( \frac{s}{2}I+X \right)_n \left( \frac{1-s}{2}I+X \right)_n }{(I+X)_n^2}=\alpha_{0}I+\alpha_{1}X, \]
where $X$ is a nilpotent matrix of order $2$. As in Proposition \ref{propo1}, we can prove that the functions $y=\alpha_{0}$ and $y=\alpha_{0} \ln (uz)+\alpha_{1}$ are linearly independent solutions of (\ref{picard1}). Then, by Proposition \ref{propo-picard-1-2}, we have that the functions $y=c \alpha_{0}^2$ are solutions of (\ref{picard2}) and this implies the relation $a_{0}=c \alpha_{0}^2$. Even more, as $a_{0}(0)=1$ and $\alpha_{0}(0)=1$, we get $a_{0}=\alpha_{0}^2$.
In the same way, as the function $y=\alpha_{0} \ln (uz)+\alpha_{1}$ is a solution of (\ref{picard1}), we have that the functions
\[
y=c \left( \alpha_{0}^2 \ln^2 (uz)+ 2 \alpha_{0} \alpha_{1} \ln (uz)+ \alpha_{1}^2 \right)
\]
are solutions of (\ref{picard2}).
As $a_{0}=\alpha_{0}^2$, comparing with
\[
y=a_{0} \frac{\ln^2 (uz)}{2}+a_{1} \ln (uz)+a_{2},
\]
we get the identities
\begin{equation}\label{a-alfa}
a_{0}=\alpha_{0}^2, \qquad a_{1}=\alpha_{0} \alpha_{1}, \qquad a_{2}=\frac{1}{2} \alpha_{1}^2.
\end{equation}
These relations imply (\ref{3F2relacion1}). To prove (\ref{3F2relacion2}), we substitute the three identities from (\ref{b11b12b13}) into it. Then, we use (\ref{3F2relacion1}) to obtain expressions for $a_2$ and $a_2'$ which allows us to complete the proof. To derive (\ref{3F2relacion3}), we first observe that (\ref{a-alfa}) imply that
\begin{equation}\label{rel-gy}
y_0=g_0^2, \qquad y_1=g_0g_1, \qquad y_2=\frac{1}{2}g_1^2,
\end{equation}
where $g_0$ and $g_1$ are the fundamental solutions of (\ref{picard1}). Differentiating we get
\begin{equation}\label{rel-diff-gy}
y'_0=2g_0g'_0, \qquad y'_1=g'_0g_1+g_0g'_1, \qquad y'_2=g_1g'_1,
\end{equation}
and
\begin{equation}\label{yp0yp1yp2}
{y'_1}-2y'_0y'_2=(g'_0g_1-g_0g'_1)^2=
\left|
  \begin{array}{cc}
    g_{0}  & g_{1}  \\
    g'_{0} & g'_{1} \\
  \end{array}
\right|^2.
\end{equation}
The function $\Phi$ and its derivative,
\begin{equation}\label{3F2-fun-Phi}
\Phi=\left|
  \begin{array}{cc}
    g_{0}  & g_{1}   \\
    g'_{0} &  g'_{1} \\
  \end{array}
\right|, \qquad
\Phi'=\left|
  \begin{array}{cc}
    g_{0}   & g_{1}   \\
    g''_{0} & g''_{1} \\
  \end{array}
\right|,
\end{equation}
satisfy
\[ \Phi'(z)=\frac{3z-2}{2z(1-z)}\Phi(z), \]
since the differential equation (\ref{picard2}) can be given in the form
\begin{equation}\label{picard2-yprimes}
g''+\frac{2-3z}{2z(1-z)}g'-\frac{s(1-s)}{4z(1-z)}g=0.
\end{equation}
Solving the equation for $\Phi(z)$, we obtain
\[ \Phi(z)=\frac{h}{z\sqrt{1-z}}, \]
where $h$ is a constant. Substituting $\Phi(z)$ in (\ref{yp0yp1yp2}) and, using (\ref{b11b12b13}), we obtain
\[ b_1^2-2b_0b_2=\frac{h^2}{1-z}. \]
Finally, substituting $z=0$, we find that $h^2=1$.
\end{proof}
\begin{lemma}
Let $M_0$, $M_1$ and $M_2$ be the determinants
\[
M_{0}=\left|
  \begin{array}{cc}
    a_{1} & b_{1} \\
    a_{2} & b_{2} \\
  \end{array}
\right|, \qquad
M_{1}=\left|
  \begin{array}{cc}
    a_{0} & b_{0} \\
    a_{2} & b_{2} \\
  \end{array}
\right|, \qquad
M_{2}=\left|
  \begin{array}{cc}
    a_{0} & b_{0} \\
    a_{1} & b_{1} \\
  \end{array}
\right|.
\]
Then we have
\begin{equation}\label{M0a2M1a1M2a0}
M_0=\frac{a_2}{\sqrt {1-z}}, \qquad M_1=\frac{a_1}{\sqrt {1-z}}, \qquad M_2=\frac{a_0}{\sqrt {1-z}}.
\end{equation}
\end{lemma}
\begin{proof}
To prove the identities for $M_0$ and $M_2$, we evaluate $M_0^2$ and $M_2^2$ and use (\ref{3F2relacion3}) to write $b_1^2$ as a function of $b_0$ and $b_2$. Then, we use (\ref{3F2relacion1}) and (\ref{3F2relacion2}) to write $a_2$ and $b_2$ as functions of $a_0$, $a_1$, $b_0$, $b_1$. To derive the identity for $M_1$, we prove that $M_1^2=2M_0M_2$ proceeding in the same way.
\end{proof}

\subsection{The equations for $z$, $a$ and $b$}
Write the matrices in (\ref{sistema1}) as the system
\begin{equation}\label{3ecu-zab}
  \begin{array}{cc}
    a_{0}a + b_{0}b =m_{0}, \\
    a_{1}a + b_{1}b =m_{1}, \\
    a_{2}a + b_{2}b =m_{2}. \\
  \end{array}
\end{equation}
As we want it to be compatible, we require
\begin{equation}\label{compatible1}
\left|
  \begin{array}{ccc}
    a_{0} & b_{0} & m_{0} \\
    a_{1} & b_{1} & m_{1} \\
    a_{2} & b_{2} & m_{2} \\
  \end{array}
\right|=0.
\end{equation}
Expanding the determinant along the last column, we obtain
\begin{equation}\label{deter2}
m_{0}M_{0}-m_{1}M_{1}+m_{2}M_{2}=0,
\end{equation}
which is an equation relating $z$ and $k$. If we define $H_0=M_{0}/M_{2}$ and $H_1=M_{1}/M_{2}$, multiply (\ref{deter2}) by $2m_0/M_2$ and define
\begin{equation}\label{tau-m0m1m2}
\tau^2=m_1^2-2m_0m_2,
\end{equation}
then equation (\ref{deter2}) takes the form
\begin{equation}\label{z-tau-ecu}
2m_0^2H_0-2m_0m_1H_1+m_1^2=\tau^2.
\end{equation}
From (\ref{M0a2M1a1M2a0}) and (\ref{3F2relacion1}) we see that
\begin{equation}
H_1=\frac{a_{1}}{a_{0}}, \qquad H_0=\frac{a_{1}^2}{2a_{0}^2},
\end{equation}
which imply that
\begin{equation}\label{3F2-H0H1-M0M1M2}
2H_0=H_1^2.
\end{equation}
This relation allows us to express equation (\ref{z-tau-ecu}) in the form
\begin{equation}
(m_1-m_0H_1)^2=\tau^2,
\end{equation}
which can be simplified to
\begin{equation}\label{ecu-simplif}
H_1-\frac{m_1}{m_0}=\frac{-\tau}{m_0}.
\end{equation}
We define the $q$-parametrization
\begin{equation}\label{qH2}
\ln(uq) = \ln(uz)+H_1-\nu_0 =\frac{-\tau}{m_0},
\end{equation}
where
\begin{equation}\label{3F2-nu0}
\nu_0=\frac{m_1}{m_0}+\ln(uz).
\end{equation}
From the definition of the matrix $M$ we derive
\begin{equation}\label{rel-eme1}
m_0=\frac{1}{\pi}, \qquad \nu_0=(\gamma+\Psi(s)-\ln 2) + (\gamma+\Psi(1-s)-\ln 2)
\end{equation}
and
\begin{equation}\label{rel-tau}
\tau=\sqrt{m_1^2-2m_0m_2}=\sqrt{k+1+ \cot^2 \pi s},
\end{equation}
where $\Psi(s)$ is the logarithmic derivative of the gamma function and $\gamma=-\Psi(1)$ is Euler's constant.
Exponentiating (\ref{qH2}) and using (\ref{rel-eme1}) and (\ref{rel-tau}), we obtain
\begin{equation}\label{todos-qz}
q=e^{-\nu_0} ze^{H_1}, \qquad q=u e^{-\pi \tau},
\end{equation}
which can be inverted to obtain $z$ as a function of $q$.
\begin{theorem}
The following formulas hold:
\begin{equation}\label{taubz}
b=\tau \sqrt{1-z},
\end{equation}
\begin{equation}\label{sol-a-y0}
y_0=\frac{q}{z \sqrt{1-z}} \cdot \frac{dz}{dq}, \qquad
a=\frac{1}{y_{0}} \left( \frac{1}{\pi} - \tau \frac{q}{y_{0}} \cdot \frac{dy_{0}}{dq} \right),
\end{equation}
\begin{equation}\label{ecu-zq-y0zq}
\frac{q}{z \sqrt{1-z}} \cdot \frac{dz}{dq} =\, _3F_2 \left(
\begin{array}{ccc}
\frac{1}{2},   & s, & 1-s \\
               & 1, & 1
\end{array}
\right| \left. z \frac{}{} \right)
\end{equation}
and (\ref{ecu-zq-y0zq}) has a unique solution $z(q)$.
\end{theorem}
\begin{proof}
If we use the identities (\ref{b11b12b13}), (\ref{3F2relacion1}), (\ref{3F2relacion2}) and (\ref{3F2relacion3}) in the calculation
\begin{align}
\tau^2=m_1^2-2m_0m_2 & =(a_1a+b_1b)^2-2(a_0a+b_0b)(a_2a+b_2b) \nonumber \\
& =(a_1^2-2a_0a_2)a^2+(b_1^2-2b_0b_2)b^2+2(a_1b_1-a_0b_2-a_2b_0)a b, \nonumber
\end{align}
we obtain the relation (\ref{taubz}). As $y_{0}=a_{0}$ and $y_{1}=a_{0} \ln (uz)+a_{1}$, from (\ref{3ecu-zab}) we can write
\begin{equation}\label{ecu2-ab}
y_{0}a + zy_{0}' b = \frac{1}{\pi}, \qquad y_{1}a + zy_{1}' b = \frac{\nu_0}{\pi}.
\end{equation}
Using (\ref{qH2}), we get the relation
\begin{equation}\label{yq}
y_{1}=y_{0}(\ln (uq)+\nu_0),
\end{equation}
which together with (\ref{ecu2-ab}), allows us to obtain the equations
\begin{equation}\label{ecu3-ab}
y_{0}a + zy_{0}' b = \frac{1}{\pi}, \qquad \left( y_{0} \ln (uq) \right) a + z \left( y_{0}\ln (uq) \right)' b = 0.
\end{equation}
From (\ref{ecu3-ab}) and (\ref{taubz}), we deduce (\ref{sol-a-y0}). Finally, as
\[
y_0=\sum_{n=0}^{\infty} \frac{\left( \frac{1}{2} \right)_n^3}{(1)_n^3} z^n(q)=\, _3F_2 \left(
\begin{array}{ccc}
\frac{1}{2},   & s, & 1-s \\
               & 1, & 1
\end{array}
\right| \left. z \frac{}{} \right)
\]
the function $z(q)$ satisfies the equation (\ref{ecu-zq-y0zq}). To see that the solution is unique, we expand $dq/q$ in powers of $z$. Integrating term by term, exponentiating and expanding again in powers of $z$, we obtain $q(z)$ as a power series in $z$. Inverting it, we obtain a unique function $z(q)$.
\end{proof}

\subsection{Algebraic and rational values}
We have prepared the way to prove some theorems related to Expansion \ref{conje1}.
\begin{theorem}
If $k$ is rational then $z$, $a$ and $b$ are algebraic.
\end{theorem}
\begin{proof}
As $H_1=a_{1}/a_{0}$, it is known that the $q$-parametrization used in (\ref{todos-qz}) is the modular one.
It is also known that if $\tau$ is a quadratic irrational then $z(\tau)$ is algebraic and therefore by (\ref{taubz}), we see that $b(\tau)$ is algebraic as well. In addition, \cite{chan2}, \cite{yangzudilin} and \cite[Sect. 3]{zudilin}, there exists an algebraic number $\delta(\tau)$, such that
\[ \frac{1}{\pi}=\delta(\tau)y_{0}(\tau)+b(\tau) z(\tau) y_{0}'(\tau), \]
where, as usual, the prime indicates that we derivate with respect to $z$. On the other hand, we have proved that
\[ \frac{1}{\pi}=a(\tau)y_{0}(\tau)+b(\tau) z(\tau) y_{0}'(\tau). \]
The comparison of the two identities gives $a(\tau)=\delta(\tau).$ We conclude that if $k$ is rational then $z(\tau)$, $b(\tau)$ and $a(\tau)$ are algebraic.
\end{proof}
\begin{theorem}
If $z$ and $b$ are algebraic, then $k$ is rational.
\end{theorem}
\begin{proof}
Since $z$ viewed as function of $i \tau$ is modular, it takes algebraic values if $i \tau$ is either a quadratic irrationality or transcendental by Schneider's theorem \cite{schneider}. The latter is impossible by (\ref{taubz}), hence $\tau^2$ is rational and so is $k$.
\end{proof}
The proofs of these theorems would have been more difficult without using the matrix equivalent of Expansion \ref{conje1} as can be specially appreciated in the developments of Sections \ref{sec-picard} and \ref{sec-ele}. On the other hand, the solutions for $z$, $a$ and $b$ provide nice matrix generalizations of the original Ramanujan's series for $1/\pi$.

\subsection{Solving the equations}
We find an explicit solution in the case $s=1/2$. Similar procedures leads to the explicit solutions in the cases $s=1/3$, $s=1/4$ and $s=1/6$.
We will use the Jacobi elliptic theta functions
\[
\theta_2(q)=\sum_{n=-\infty}^{\infty} q^{{(n+1/2)}^2}, \quad
\theta_3(q)=\sum_{n=-\infty}^{\infty} {q^{n^2}}, \quad
\theta_4(q)=\sum_{n=-\infty}^{\infty} (-1)^n{q^{n^2}},
\]
the elliptic lambda function and the elliptic alpha function
\begin{equation}\label{lambda-alpha}
\lambda (q)=\frac{\theta_2^4(q)}{\theta_3^4(q)}, \qquad \alpha(q)=\frac{1}{\theta_3^4(q)}
\left( \frac{1}{\pi} - 4 \tau \frac{q}{\theta_4(q)} \cdot \frac{d \theta_4(q)}{dq} \right),
\end{equation}
the formula (see \cite{borweinagm}, p. 35)
\begin{equation}\label{rel-thetas}
\theta_3^4(q)=\theta_2^4(q)+\theta_4^4(q)
\end{equation}
and the following two formulas (see \cite{borweinagm}, p. 42):
\begin{equation}\label{for-bor}
\frac{4q}{\theta_3(q)} \cdot \frac{d \theta_3(q)}{dq}-\frac{4q}{\theta_4(q)} \cdot
\frac{d \theta_4(q)}{dq}=\theta_2^4(q),
\end{equation}
\begin{equation}\label{for-bor2}
\frac{4q}{\theta_2(q)} \cdot \frac{d \theta_2(q)}{dq}-\frac{4q}{\theta_3(q)} \cdot
\frac{d \theta_3(q)}{dq}=\theta_4^4(q).
\end{equation}
For $s=1/2$ we have proved that $\tau=\sqrt{k+1}$.
Consider the function
\begin{equation}\label{z-sol-modular}
z(q)=4 \cdot \frac{\theta_2^4(q)}{\theta_3^4(q)} \cdot \frac{\theta_4^4(q)}{\theta_3^4(q)}. \end{equation}
Taking logarithms in (\ref{z-sol-modular}), differentiating with respect to $q$ and using (\ref{for-bor}), (\ref{for-bor2}), we see that
\begin{equation}\label{a-lambda-alpha}
\frac{q}{z \sqrt{1-z}} \cdot \frac{dz}{dq}=\theta_3^4(q).
\end{equation}
But it is well known \cite[p. 180]{borweinagm} that
\begin{equation}\label{ecu-ztheta-y0theta}
\theta_3^4(q) =\, _3F_2
\left(
\begin{array}{ccc}
\frac{1}{2},   & \frac{1}{2}, & \frac{1}{2} \\
               & 1,           & 1
\end{array}
\right|
\left. 4 \cdot \frac{\theta_2^4(q)}{\theta_3^4(q)} \cdot \frac{\theta_4^4(q)}{\theta_3^4(q)} \right).
\end{equation}
For a general method to prove identities like (\ref{ecu-ztheta-y0theta}) see \cite{yang}. So, (\ref{z-sol-modular}) is the unique solution of (\ref{ecu-zq-y0zq}).
Writing it with the function $\lambda(q)$ defined in (\ref{lambda-alpha}), we have
\begin{equation}\label{z-sol-lambda}
z(q)=4 \lambda(q)(1-\lambda(q)).
\end{equation}
Then, using the formulas (\ref{taubz}) and (\ref{z-sol-modular}), we find that
\begin{equation}\label{b-lambda}
b=\sqrt{k+1} \cdot (2 \lambda(q)-1).
\end{equation}
Finally, using (\ref{for-bor}), (\ref{sol-a-y0}) and the definitions in (\ref{lambda-alpha}), we obtain
\begin{equation}\label{a-lambda-alpha}
a=\alpha(q)-\sqrt{k+1} \cdot \lambda(q).
\end{equation}

We know, from the theory of elliptic modular functions, that for rational values of $k$ the functions $\lambda(q)$ and $\alpha(q)$ take algebraic values and, therefore, if $k$ is rational, then $z$, $a$ and $b$ are algebraic.

\section{Matrix form of  Expansion \ref{conje2}}\label{5F4}

We obtain a matrix equivalent of Expansion \ref{conje2} replacing $x$ with a nilpotent matrix $X$ of order $5$. This equivalence is a consequence of Lemma \ref{lema}. We introduce the notation
\[ P_n(s,t,X)=\frac{ \left( \frac{1}{2}I+X \right)_n (sI+X)_n (tI+X)_n ((1-t)I+X)_n ((1-s)I+X)_n}{ (I+X)_n^5}; \]
using the property $(s)_{n+x}=(s+x)_n(s)_x$ we obtain
\begin{expansion}[Expansion \ref{conje2} in matrix form]\label{matrix5F4}
\begin{multline}\nonumber
\sum_{n=0}^{\infty} z^n  P_n(s,t,X) \left( \frac{}{} \! aI+b(nI+X)+c(nI+X)^2 \right) \\ = (uz)^{-X} P_X^{-1}(s,t,0) \left(\frac{1}{\pi^2}-\frac{k}{2}X^2+\frac{j}{24}\pi^2 X^4 \right),
\end{multline}
where $X$ is an arbitrary nilpotent matrix of order $5$.
\end{expansion}
Denote
\[
A=\sum_{n=0}^{\infty} z^n P_n(s,t,X), \qquad B = \sum_{n=0}^{\infty} z^n P_n(s,t,X)(nI+X),
\]
\[
C=\sum_{n=0}^{\infty} z^n P_n(s,t,X)(nI+X)^2
\]
and
\[
M=(uz)^{-X} P_X^{-1}(s,t,0) \left( \frac{1}{\pi^2} I -\frac{k}{2} X^2 + \frac{j}{24} \pi^2 X^4 \right).
\]
Then Expansion \ref{matrix5F4} can be written in the form $aA+bB+cC=M$, which is equivalent to the (overdetermined) system
\begin{equation}\label{5ecuaciones}
\left(
  \begin{array}{ccc}
    a_{0} & b_{0} & c_{0} \\
    a_{1} & b_{1} & c_{1} \\
    a_{2} & b_{2} & c_{2} \\
    a_{3} & b_{3} & c_{3} \\
    a_{4} & b_{4} & c_{4} \\
  \end{array}
\right)
\left(
  \begin{array}{c}
    a \\
    b \\
    c \\
  \end{array}
\right)=
\left(
  \begin{array}{c}
    m_{0} \\
    m_{1} \\
    m_{2} \\
    m_{3} \\
    m_{4} \\
  \end{array}
\right).
\end{equation}

\subsection{Picard-Fuchs differential equation}\label{sec-pic}
We will see that the components of the matrix $Y=(uz)^X A$, namely
\begin{align}
y_0= & a_0, \nonumber \\
y_1= & a_0 \ln(uz)+a_1, \nonumber \\
y_2= & a_0 \frac{\ln^2(uz)}{2}+a_1 \ln(uz)+a_2, \label{ysub5F4} \\
y_3= & a_0 \frac{\ln^3(uz)}{6}+a_1 \frac{\ln^2(uz)}{2}+a_2 \ln(uz)+a_3, \nonumber \\
y_4= & a_0 \frac{\ln^4(uz)}{24}+a_1 \frac{\ln^3(uz)}{6}+a_2 \frac{\ln^2(uz)}{2}+a_3 \ln(uz)+a_4, \nonumber
\end{align}
satisfy a Picard-Fuchs differential equation by proving that the matrix $Y$ itself satisfies that equation.

\begin{proposition}\label{propo2}
The matrix $Y=(uz)^X A$ is a fundamental solution of the differential equation
\begin{equation}\label{pi-fu}
\left( z\frac{d}{dz} \right)^5 \! Y \! = \! z \left( z\frac{d}{dz}+\frac{1}{2} \right) \left( z\frac{d}{dz}+s \right) \left( z\frac{d}{dz}+t \right) \left( z\frac{d}{dz}+1-t \right) \left( z\frac{d}{dz}+1-s \right) Y.
\end{equation}
\end{proposition}
\begin{proof}
Completely analogous to the above proof of  Proposition \ref{propo1}.
\end{proof}

\subsection{The components of the matrices $A$, $B$ and $C$}
There are many relations among the components of the matrices $A$, $B$ and $C$.
\begin{lemma}
The following identities hold:
\begin{equation}\label{compo-AB}
b_0=za_0', \quad b_1=a_0+z a_1', \quad b_2=a_1+za_2', \quad b_3=a_2+za_3', \quad b_4=a_3+za_4',
\end{equation}
and
\begin{equation}\label{compo-BC}
c_0=zb_0', \quad c_1=b_0+z b_1', \quad c_2=b_1+zb_2', \quad c_3=b_2+zb_3', \quad c_4=b_3+zb_4'.
\end{equation}
\end{lemma}
\begin{proof}
The matrices $B$ and $C$ satisfy the relations
\[ B=zA'+XA, \qquad C=zB'+XB, \]
whose components imply (\ref{compo-AB}) and (\ref{compo-BC}), respectively.
\end{proof}
\begin{lemma}
The following relations hold:
\begin{equation}\label{paren-1}
2a_0a_4-2a_1a_3+a_2^2=0,
\end{equation}
\begin{equation}\label{paren-2}
2b_0b_4-2b_1b_3+b_2^2=0,
\end{equation}
\begin{equation}\label{paren-3}
2c_0c_4-2c_1c_3+c_2^2=\frac{1}{1-z}.
\end{equation}
\end{lemma}
\begin{proof}
It is known \cite[Proposition 4]{almkvist} that we can write the functions in (\ref{ysub5F4}) in the form
\begin{align} \label{calabi-yg}
y_0=z & \left|
  \begin{array}{cc}
    g_{0}  & g_{1}  \\
    g'_{0} & g'_{1} \\
  \end{array}
\right|, \quad
y_1= z \left|
  \begin{array}{cc}
    g_{0}  & g_{2}  \\
    g'_{0} & g'_{2} \\
  \end{array}
\right|, \quad
y_2=z \left|
  \begin{array}{cc}
    g_{0}  & g_{3}  \\
    g'_{0} & g'_{3} \\
  \end{array}
\right|, \nonumber \\
& y_3= \frac{z}{2} \left|
  \begin{array}{cc}
    g_{1}  & g_{3}  \\
    g'_{1} & g'_{3} \\
  \end{array}
\right|, \qquad
y_4=\frac{z}{2} \left|
  \begin{array}{cc}
    g_{2}  & g_{3}  \\
    g'_{2} & g'_{3} \\
  \end{array}
\right|,
\end{align}
with
\begin{equation}\label{y2-twoways}
\left|
  \begin{array}{cc}
    g_{0}  & g_{3}  \\
    g'_{0} & g'_{3} \\
  \end{array}
\right|=
\left|
  \begin{array}{cc}
    g_{1}  & g_{2}  \\
    g'_{1} & g'_{2} \\
  \end{array}
\right|,
\end{equation}
where $g_0$, $g_1$, $g_2$, $g_3$ are solutions of a certain fourth order linear differential equation. Differentiating twice (\ref{y2-twoways}), we obtain the relations
\begin{equation}\label{dy2-twoways}
\left|
  \begin{array}{cc}
    g_{0}   & g_{3}   \\
    g''_{0} & g''_{3} \\
  \end{array}
\right|=\left|
  \begin{array}{cc}
    g_{1}   & g_{2}   \\
    g''_{1} & g''_{2} \\
  \end{array}
\right|,
\end{equation}
\begin{equation}\label{ddy2-twoways}
\left|
  \begin{array}{cc}
    g'_{0}   & g'_{3}   \\
    g''_{0} & g''_{3} \\
  \end{array}
\right|+\left|
  \begin{array}{cc}
    g_{0}   & g_{3}   \\
    g'''_{0} & g'''_{3} \\
  \end{array}
\right|=\left|
  \begin{array}{cc}
    g'_{1}   & g'_{2}   \\
    g''_{1} & g''_{2} \\
  \end{array}
\right|+\left|
  \begin{array}{cc}
    g_{1}   & g_{2}   \\
    g'''_{1} & g'''_{2} \\
  \end{array}
\right|.
\end{equation}
We will need (\ref{dy2-twoways}) in the proof of (\ref{paren-2}), and both identities in the proof of (\ref{paren-3}).
Using (\ref{ysub5F4}), we see that the identities (\ref{paren-1}) and (\ref{paren-2}) are an immediate consequence of the relations, proved in the corollary of \cite[Proposition 4]{almkvist}:
\begin{equation}\label{yy-diffydiffy}
2y_0y_4-2y_1y_3+y_2^2=0, \qquad 2y'_0y'_4-2y'_1y'_3+{y'_2}^2=0.
\end{equation}
We recall those proofs because we will need them to derive (\ref{paren-3}). Writing
\[
 y_2^2= z \left|
  \begin{array}{cc}
    g_{0}  & g_{3}  \\
    g'_{0} & g'_{3} \\
  \end{array}
\right| \cdot
z \left|
  \begin{array}{cc}
    g_{1}  & g_{2}  \\
    g'_{1} & g'_{2} \\
  \end{array}
\right|,
\]
we see that the first identity in (\ref{yy-diffydiffy}) is trivial by expanding. In the same vein, differentiating (\ref{calabi-yg}), we deduce that
\[ 2 \left(y'_0-\frac{y_0}{z}\right) \left(y'_4-\frac{y_4}{z}\right)-
2\left(y'_1-\frac{y_1}{z}\right) \left(y'_3-\frac{y_3}{z}\right)+\left(y'_2-\frac{y_2}{z}\right)^2=0,  \]
which, together with the first identity in (\ref{yy-diffydiffy}), implies the second identity in (\ref{yy-diffydiffy}).
To prove (\ref{paren-3}) we differentiate (\ref{calabi-yg}) twice and use the relations (\ref{y2-twoways}), (\ref{dy2-twoways}), (\ref{ddy2-twoways}), to obtain
\begin{align}\label{ddy2-Phi}
\left(y''_2-\frac{2y'_2}{z}+\frac{2y_2}{z^2} \right)^2- & 2\left (y''_1-\frac{2y'_1}{z}+\frac{2y_1}{z^2} \right) \left(y''_3-\frac{2y'_3}{z}+\frac{2y_3}{z^2} \right) \nonumber \\ + & 2\left(y''_0-\frac{2y'_0}{z}+\frac{2y_0}{z^2} \right)\left(y''_4-\frac{2y'_4}{z}+\frac{2y_4}{z^2} \right)=z^2 \Phi^2,
\end{align}
where
\begin{equation}\label{fun-Phi}
\Phi=
\left|
  \begin{array}{cc}
    g_{0}   &  g_{3}    \\
    g'''_{0} & g'''_{3} \\
  \end{array}
\right| -
\left|
  \begin{array}{cc}
    g_{1}    & g_{2}    \\
    g'''_{1} & g'''_{2} \\
  \end{array}
\right|=
\left|
  \begin{array}{cc}
    g'_{1}  & g'_{2}  \\
    g''_{1} & g''_{2} \\
  \end{array}
\right|-
\left|
  \begin{array}{cc}
    g'_{0}  & g'_{3}  \\
    g''_{0} & g''_{3} \\
  \end{array}
\right|.
\end{equation}
Using (\ref{yy-diffydiffy}) we can simplify the left hand-side of (\ref{ddy2-Phi}), and we obtain
\begin{equation}\label{ddy-Phi2}
{y''_2}^2-2y''_1y''_3+2y''_0y''_4=z^2\Phi^2.
\end{equation}
Summing the two identities for $\Phi$ in (\ref{fun-Phi}) and differentiating, we have
\begin{equation}\label{sum-fun-Phi}
2 \Phi'=
\left|
  \begin{array}{cc}
    g_{0}     & g_{3}     \\
    g''''_{0} & g''''_{3} \\
  \end{array}
\right| -
\left|
  \begin{array}{cc}
    g_{1}     & g_{2}     \\
    g''''_{1} & g''''_{2} \\
  \end{array}
\right|.
\end{equation}
But for all of the $14$ hypergeometric equations (\ref{pi-fu}) we have from \cite[Subsect. 2.1]{almkvist0} that the fourth order linear differential equation pullback is of the form
\begin{equation}\label{diff-eq-calabi}
g''''+\frac{6-7z}{z(1-z)} g'''+\rho_2(z) g''+\rho_1(z) g'+\rho_0(z) g=0.
\end{equation}
Therefore, from (\ref{y2-twoways}), (\ref{dy2-twoways}), (\ref{ddy2-twoways}), (\ref{fun-Phi}), (\ref{sum-fun-Phi}) and (\ref{diff-eq-calabi}), we deduce that
\[ 2\Phi'(z)=\frac{7z-6}{z(1-z)} \Phi(z). \]
Solving this equation we obtain
\[ \Phi(z)=h \cdot \frac{1}{z^3 \sqrt{1-z}}, \]
where $h$ is a constant. Thus, from (\ref{ddy-Phi2}), we get
\[ {y''_2}^2-2y''_1y''_3+2y''_0y''_4=\frac{h^2}{z^4(1-z)}, \]
which implies that
\[ c_2^2-2c_1c_3+2c_0c_4=\frac{h^2}{1-z}. \]
Substituting $z=0$, we finally determine that $h^2=1$.
\end{proof}

\begin{lemma}
The following relations hold:
\begin{equation}\label{paren-4}
a_0b_4+a_4b_0-a_1b_3-a_3b_1+a_2b_2=0,
\end{equation}
\begin{equation}\label{paren-5}
a_0c_4+a_4c_0-a_1c_3-a_3c_1+a_2c_2=0,
\end{equation}
\begin{equation}\label{paren-6}
b_0c_4+b_4c_0-b_1c_3-b_3c_1+b_2c_2=0.
\end{equation}
\end{lemma}
\begin{proof}
To prove (\ref{paren-4}), we apply the operator $z\frac{d}{dz}$ to (\ref{paren-1}), and use (\ref{compo-AB}) to substitute $za'_0=b_0$, $za'_1=b_1-a_0$, $za'_2=b_2-a_1$, $za'_3=b_3-a_2$, $za'_4=b_4-a_3$. To prove (\ref{paren-6}) we apply  the operator to (\ref{paren-2}) and use (\ref{compo-BC}). Finally, to prove (\ref{paren-5}), we apply the operator to (\ref{paren-4}) and use (\ref{compo-AB}), (\ref{compo-BC}) and (\ref{paren-2}).
\end{proof}

\begin{lemma}
Let $u_j$, $v_j$ and $w_j$ be the determinants
\begin{equation}\label{def-ujvjwj}
u_j=\left|
  \begin{array}{cc}
    a_{0} & b_{0} \\
    a_{j} & b_{j} \\
  \end{array}
\right|, \qquad
v_j= \left|
  \begin{array}{cc}
    a_{0} & c_{0} \\
    a_{j} & c_{j} \\
  \end{array}
\right|, \qquad
w_j=\left|
  \begin{array}{cc}
    b_{0} & c_{0} \\
    b_{j} & c_{j} \\
  \end{array}
\right|.
\end{equation}
Then, the following identities hold:
\begin{equation}\label{ecu-uj}
u_2^2-2u_1u_3=0,
\end{equation}
\begin{equation}\label{1ecu1-z}
v_2^2-2v_1v_3=\frac{a_0^2}{1-z},
\end{equation}
\begin{equation}\label{2ecu1-z}
w_2^2-2w_1w_3=\frac{b_0^2}{1-z},
\end{equation}
\begin{equation}\label{ecu-uj-vj}
u_2v_2-u_1v_3-v_1u_3=0,
\end{equation}
\begin{equation}\label{ecu-uj-wj}
u_2w_2-u_1w_3-w_1u_3=0,
\end{equation}
\begin{equation}\label{ecu-vj-wj}
v_2w_2-v_1w_3-w_1v_3=\frac{a_0b_0}{1-z}.
\end{equation}
\end{lemma}
\begin{proof}
To prove (\ref{ecu-uj}), we use  (\ref{paren-1}) and (\ref{paren-4}) to write $a_4$, $b_4$ as functions of $a_0$, $b_0$, $a_1$, $b_1$, $a_2$, $b_2$, $a_3$, $b_3$; then, we substitute $b_4$ in (\ref{paren-2}) and simplify. To prove (\ref{1ecu1-z}), we use (\ref{paren-1}) and (\ref{paren-5}) to write $a_4$, $c_4$ as functions of $a_0$, $c_0$, $a_1$, $c_1$, $a_2$, $c_2$, $a_3$, $c_3$; then, we substitute $c_4$ in (\ref{paren-3}) and simplify. In the same way, to derive (\ref{2ecu1-z}), we use (\ref{paren-2}) and (\ref{paren-6}); then, we substitute $c_4$ in (\ref{paren-3}) and simplify. Finally, to derive (\ref{ecu-uj-vj}), (\ref{ecu-uj-wj}) and (\ref{ecu-vj-wj}), we use (\ref{paren-4}), (\ref{paren-5}), (\ref{paren-6}) to write $a_4$, $b_4$, $c_4$ as functions of $a_0$, $b_0$, $c_0$, $a_1$, $b_1$, $c_1$, $a_2$, $b_2$, $c_2$, $a_3$, $b_3$, $c_3$; then, we substitute in (\ref{paren-1}), (\ref{paren-2}) and (\ref{paren-3}), respectively.
\end{proof}

\begin{lemma}
Let $M_0$, $M_1$, $M_2$ and $M_3$ be the determinants
\begin{equation}\label{deter-M0-M1}
M_0=\left|
  \begin{array}{ccc}
    a_{1} & b_{1} & c_{1} \\
    a_{2} & b_{2} & c_{2} \\
    a_{3} & b_{3} & c_{3} \\
  \end{array}
\right|, \qquad
M_1=\left|
  \begin{array}{ccc}
    a_{0} & b_{0} & c_{0} \\
    a_{2} & b_{2} & c_{2} \\
    a_{3} & b_{3} & c_{3} \\
  \end{array}
\right|,
\end{equation}
\begin{equation}\label{deter-M2-M3}
M_2=\left|
  \begin{array}{ccc}
    a_{0} & b_{0} & c_{0} \\
    a_{1} & b_{1} & c_{1} \\
    a_{3} & b_{3} & c_{3} \\
  \end{array}
\right|, \qquad
M_3=\left|
  \begin{array}{ccc}
    a_{0} & b_{0} & c_{0} \\
    a_{1} & b_{1} & c_{1} \\
    a_{2} & b_{2} & c_{2} \\
  \end{array}
\right|.
\end{equation}
Then, we have
\begin{equation}\label{M1M2M3}
M_3=\frac{u_1}{\sqrt{1-z}}, \quad M_2=\frac{u_2}{\sqrt{1-z}}, \quad
M_1=\frac{u_3}{\sqrt{1-z}}, \quad M_0=\frac{u_4}{\sqrt{1-z}}.
\end{equation}
\end{lemma}
\begin{proof}
To prove the first identity, we use (\ref{ecu-uj}), (\ref{ecu-uj-vj}) and (\ref{ecu-uj-wj})  to write $u_3$, $v_3$, $w_3$ as functions of $u_0$, $v_0$, $w_0$, $u_1$, $v_1$, $w_1$, $u_2$, $v_2$, $w_2$. Then, we substitute $v_3$ in (\ref{1ecu1-z}) and simplify. To derive the other identities, we prove that
\begin{equation}\label{MoverM}
\frac{M_0}{M_3}=\frac{u_4}{u_1}, \qquad \frac{M_1}{M_3}=\frac{u_3}{u_1}, \qquad \frac{M_2}{M_3}=\frac{u_2}{u_1}.
\end{equation}
For it we use  (\ref{ecu-uj}), (\ref{ecu-uj-vj}) and (\ref{paren-1}), (\ref{paren-2}), (\ref{paren-5}) to write $a_4$, $b_4$, $c_4$, $a_3$, $b_3$ as functions of $a_0$, $b_0$, $c_0$, $a_1$, $b_1$, $c_1$, $a_2$, $b_2$, $c_2$ and $c_3$.  Then, we substitute these values in the identity we want to prove and simplify.
\end{proof}

\begin{lemma}
The following identity holds:
\begin{equation}\label{paraJ}
\frac{u_1v_4-u_4v_1}{u_1v_2-u_2v_1}=\frac{a_1b_2-a_2b_1}{a_0b_1-a_1b_0}.
\end{equation}
\end{lemma}
\begin{proof}
Use  (\ref{ecu-uj}), (\ref{ecu-uj-vj}) and (\ref{paren-1}), (\ref{paren-2}), (\ref{paren-5}) to write $a_4$, $b_4$, $c_4$, $a_3$, $b_3$ as functions of $a_0$, $b_0$, $c_0$, $a_1$, $b_1$, $c_1$, $a_2$, $b_2$, $c_2$ and $c_3$. Then simplify.
\end{proof}

\subsection{The equations for $z$, $a$, $b$ and $c$}
As we want the system (\ref{5ecuaciones}) to be compatible we first impose that
\begin{equation}\label{deter-eq0}
\left|
  \begin{array}{cccc}
    a_{0} & b_{0} & c_{0} & m_{0} \\
    a_{1} & b_{1} & c_{1} & m_{1} \\
    a_{2} & b_{2} & c_{2} & m_{2} \\
    a_{3} & b_{3} & c_{3} & m_{3} \\
  \end{array}
\right|=0.
\end{equation}
Expanding the determinant along the last column, we obtain
\begin{equation}\label{des-deter}
m_{0}M_{0}-m_{1}M_{1}+m_{2}M_{2}-m_{3}M_{3}=0.
\end{equation}
We now define the functions $H_0=M_{0}/M_{3}$ and $H_1=M_{1}/M_{3}$ and $H_2=M_{2}/M_{3}$. Relations (\ref{MoverM}), imply the identities
\begin{equation}\label{H2H1H0}
H_2=\frac{a_0b_2-a_2b_0}{a_0b_1-a_1b_0}, \qquad  H_1=\frac{a_0b_3-a_3b_0}{a_0b_1-a_1b_0}, \qquad H_0=\frac{a_0b_4-a_4b_0}{a_0b_1-a_1b_0},
\end{equation}
and (\ref{ecu-uj}) implies $2H_1=H_2^2$. This last relation allows us to simplify the equation (\ref{des-deter}), and we obtain
\begin{equation}\label{ecu5F4emes}
\frac{1}{6}\left( H_2+\ln(uz)-\nu_0 \right)^3-\nu_1\left( H_2+\ln(uz)-\nu_0 \right)-
\nu_2-\left( \frac{1}{6}H_2^3-H_0  \right)=0,
\end{equation}
where
\begin{equation}\label{nu0nu1nu2}
\nu_0=\frac{m_1}{m_0}+\ln(uz), \qquad
\nu_1=\frac{m_1^2}{2m_0^2}-\frac{m_2}{m_0}, \qquad \nu_2=\frac{m_1^3}{3m_0^3}-\frac{m_1m_2}{m_0^2}+\frac{m_3}{m_0}.
\end{equation}
From the definition of the matrix $M$ we derive
\begin{equation}\label{5F4-m0}
m_0=\frac{1}{\pi^2},
\end{equation}
and
\begin{equation}\label{nu0}
\nu_0=(\Psi(s)+\Psi(1-s)+2\gamma-\ln 2)+(\Psi(t)+\Psi(1-t)+2\gamma-\ln 2),
\end{equation}
\begin{equation}\label{nu1}
\nu_1=\frac{\pi^2}{2} \left( k+\frac{5}{3}+\cot^2 \pi s + \cot^2 \pi t \right),
\end{equation}
\begin{equation}\label{nu2}
\nu_2=\frac{1}{6}(4\zeta(3)-\Psi''(s)-\Psi''(1-s)-\Psi''(t)-\Psi''(1-t)).
\end{equation}
With the $q$-parametrization
\begin{equation}\label{5F4-lnq-m0m1}
\ln(uq)=H_2+\ln(uz)-\nu_0,
\end{equation}
equation (\ref{ecu5F4emes}) can be written as
\begin{equation}\label{ecu-lnq}
\frac{1}{6} \ln^3 (uq) -\nu_1 \ln (uq)-\nu_2-T(q)=0,
\end{equation}
where $T=\frac{1}{6}H_2^3-H_0$. On the other hand, we can obtain more results using all the equations of the system (\ref{5ecuaciones}), that is,
\begin{equation}\label{5ecu}
  \begin{array}{cc}
    a_{0}a + b_{0}b + c_{0}c=m_{0}, \\
    a_{1}a + b_{1}b + c_{1}c=m_{1}, \\
    a_{2}a + b_{2}b + c_{2}c=m_{2}, \\
    a_{3}a + b_{3}b + c_{3}c=m_{3}, \\
    a_{4}a + b_{4}b + c_{4}c=m_{4}. \\
  \end{array}
\end{equation}
The idea consists in the following calculation:
\begin{multline}\label{multiply-ecus}
2m_0m_4-2m_1m_3+m_2^2 = (2a_0a_4-2a_1a_3+a_2^2)a^2 + (2b_0b_4-2b_1b_3+b_2^2)b^2 + \\ (2c_0c_4-2c_1c_3+c_2^2)c^2 + 2(a_0b_4+a_4b_0-a_1b_3-a_3b_1+a_2b_2)a b +  \\ 2(a_0c_4+a_4c_0-a_1c_3-a_3c_1+a_2c_2)ac + 2(b_0c_4+b_4c_0-b_1c_3-b_3c_1+b_2c_2)bc.
\end{multline}
If we define
\[ \tau^2=2m_0m_4-2m_1m_3+m_2^2, \]
then substituting identities (\ref{paren-1}), (\ref{paren-2}), (\ref{paren-4}), (\ref{paren-6}), (\ref{paren-5}), (\ref{paren-3}) in (\ref{multiply-ecus}) we obtain
\begin{equation}\label{cz-tau}
\tau^2=\frac{c^2}{1-z},
\end{equation}
The definition of the matrix $M$ implies that
\begin{equation}\label{tau-al-cua}
\tau^2=\frac{j}{12}+\frac{k^2}{4}+\frac{5k}{3}+1+
(\cot^2 \pi s)(\cot^2 \pi t)+(1+k)(\cot^2 \pi s +\cot^2 \pi t).
\end{equation}
\par Relation (\ref{deter-eq0}) is a necessary but not sufficient condition for the system (\ref{5ecu}) to be compatible. In order to have a necessary and sufficient condition for it, we also impose that the equation for $c$ given in (\ref{cz-tau}) is compatible with the three first equations in (\ref{5ecu}).
Solving $c$ from them by Cramer's rule, we get
\[ c=m_0 \frac{a_1b_2-a_2b_1}{M_3}-m_1 \frac{a_0b_2-a_2b_0}{M_3}+m_2 \frac{a_0b_1-a_1b_0}{M_3}. \]
Defining the function $J=(a_1b_2-a_2b_1)/(a_0b_1-a_1b_0)$, and using identities (\ref{M1M2M3}), we get
$c=(m_0 J-m_1 H_2+m_2) \sqrt{1-z}$. From this we obtain
\[ \tau=m_0 J-m_1H_2+m_2. \]
Using (\ref{nu0nu1nu2}), the latter can be rearranged as
\begin{equation}\label{5F4-ecu-tau}
\frac{\tau}{m_0}+\nu_1=\frac{1}{2} \left( H_2+\ln(uz)-\nu_0 \right)^2-\left( \frac{1}{2}H_2^2-J \right).
\end{equation}
Finally, the parametrization in (\ref{5F4-lnq-m0m1}) gives
\begin{equation}\label{5F4-qU}
\pi^2 \tau+\nu_1=\frac{1}{2} \ln^2(uq)-U(q),
\end{equation}
where $U=\frac{1}{2}H_2^2-J=H_1-J$. For a given value of $k$, the equation (\ref{ecu-lnq}) determines $q$. The substitution of these values of $k$ and $q$ into (\ref{5F4-qU}) determines $\tau$ and $j$. In addition, $q$ determines $z$ from (\ref{5F4-lnq-m0m1}), while $q$ and $\tau$ determine (\ref{cz-tau}). Substituting the values for $z$ and $c$ in the first two equations of (\ref{5ecu}), we obtain $a$ and $b$.

\begin{proposition}
The functions $T$ and $U$ are related by
\begin{equation}\label{UT}
U(q)= q \frac{d}{dq} T(q).
\end{equation}
\end{proposition}
\begin{proof}
Differentiating (\ref{5F4-lnq-m0m1}), we get
\[ q \frac{dz}{dq}=\frac{z}{1+zH'_2}. \]
So, we have
\[
\frac{1}{2}H_2^2 - q \frac{dT}{dq}= H_1 -q T' \frac{dz}{dq}=\frac{H_1+zH'_0}{1+zH'_2}=\frac{u_1v_4-u_4v_1}{u_1v_2-u_2v_1}.
\]
To complete the proof, we use (\ref{paraJ}).
\end{proof}

\begin{corollary}
The functions $k$ and $\tau$ are related by
\[ \tau=\frac{q\ln uq}{2} \cdot \frac{dk}{dq}.  \]
\end{corollary}
\begin{proof}
Differentiate (\ref{ecu-lnq}) and compare to (\ref{5F4-qU}). Then, use (\ref{nu1}).
\end{proof}

\subsection{Algebraic and rational values}
We state a theorem and make some conjectures related to Expansion \ref{conje2}.
\begin{theorem}
If $z$ and $c$ are algebraic, then $\tau$ is algebraic.
\end{theorem}
\begin{proof}
Immediate from (\ref{cz-tau}).
\end{proof}
\begin{conjecture}
If $z$ and $c$ are algebraic, then $\tau$ is a quadratic irrational.
\end{conjecture}
\begin{conjecture}
If $z$ and $c$ are algebraic, then $k$ and $j$ are rational.
\end{conjecture}
\begin{conjecture}
If $k$ and $j$ are rational, then $z$, $a$, $b$ and $c$ are algebraic.
\end{conjecture}

\subsection{Some computations}
For our computations we have chosen the nilpotent matrix $X$ of order $5$ such that all its entries are zero except $x_{12}=x_{23}=x_{34}=x_{45}=1$. With this choice the components of the matrices $A$, $B$ and of $C$ are the entries of their first rows. Here we present the computations in the case $s=t=1/2$.
The $q$-parametrization in (\ref{5F4-lnq-m0m1}) in this case is
\begin{equation}\label{5F4-lnq}
\ln(uq)=-10 \ln 2 +\ln (uz) + H_2.
\end{equation}
By exponentiation, we get
\begin{equation}\label{q5F4z}
q=\frac{1}{1024} z e^{H_2},
\end{equation}
where
\begin{equation}\label{expH2-taylor}
e^{H_2}=1+320 \left( \frac{z}{2^{10}} \right)+170400 \left(\frac{z}{2^{10}}\right)^2+110694400 \left(\frac{z}{2^{10}}\right)^3+\cdots.
\end{equation}
By inversion of (\ref{q5F4z}), we obtain
\begin{equation}\label{mirrorz}
z(q)=1024(q-320q^2+34400q^3-1894400q^4+62019120q^5-\cdots).
\end{equation}
And the function $T(q)$, that we have computed for this case, is
\begin{equation}\label{Tq}
T(q)=160\left(q+\frac{347}{2^3}q^2+\frac{91072}{3^3}q^3+\frac{21827771}{4^3}q^4+\frac{5002311376}{5^3}q^5+\cdots \right).
\end{equation}
We have solved the equations of the preceding section numerically and "identified" algebraic solutions when $k=1$ and $k=5$.
They are
\begin{equation}\label{algebraic1}
k=1, \quad \tau=\sqrt{5}, \quad j=25, \quad z=-\frac{1}{4}, \quad a=\frac{1}{8}, \quad b=1, \quad c=\frac{5}{2}
\end{equation}
and
\begin{equation}\label{algebraic2}
k=5, \quad \tau=\sqrt{41}, \quad j=305, \quad z=-\frac{1}{1024}, \quad a=\frac{13}{128}, \quad b=\frac{45}{32},
\quad c=\frac{205}{32}.
\end{equation}
\par We know that in the families $(s,t)$: $(1/2, 1/4)$, $(1/4, 1/6)$, $(1/4, 1/3)$, $(1/3, 1/6)$ and $(1/8, 3/8)$ there are also Ramanujan-like series (some of them only conjectured), all given in \cite{guilleraEMpi2}, and we have not found any in the other families.

\section{Conclusion}
Many steps of our theoretical development have been suggested by experimental computations carried over using nilpotent matrices. Following this idea, our symbolic calculations with Maple 9 allowed us to discover important relations among the entries of the matrices $A$ and $B$ (Sect. 2) and $A$, $B$ and $C$ (Sect. 3). Later, we found the proofs of them given in this paper.
\par
We have solved completely the study of the expansions associated to the Rama\-nujan-type series for $1/\pi$
(type $_3F_2$) because they lead to the well-known theory of elliptic modular functions.
Concerning the analysis of the expansions associated to the Ramanujan-like series for $1/\pi^2$ (type $_5F_4$) we point out the following interesting connection with the Calabi--Yau differential equations. If $s=t=1/2$, the expansions of $z(q)$ and $K(q)=-1+(q \frac{d}{dq})^2 U$ coincide, respectively, with those obtained from the definitions of the mirror map and the Yukawa coupling given in \cite{yangzudilin} for the same case.
A difficult aspect seems to be determining rigourously the values of $j$ and $k$ which lead to Ramanujan-like series for $1/\pi^2$. Related to this, we believe that the coefficients of the Humbert surfaces found experimentally in \cite[Sect. 6]{yangzudilin} depend only on the values of $s$, $t$ and $j$, $k$. In addition, we observe, that the value of $\tau^2$ coincide with the discriminant of those Humbert surfaces.
\par
Even more difficult will certainly be the analysis of the expansions corresponding to higher degree series.
The only known example is
\begin{equation}\label{boris}
\frac{1}{32} \sum_{n=0}^{\infty}  \frac{\left(\frac{1}{2}\right)_n^7}{(1)_n^7} \frac{1}{2^{6n}} (168 n^3+76 n^2+14n+1)={1 \over \pi^3},
\end{equation}
which remains unproved and was discovered by B. Gourevitch \cite{guilleraEMpi2}.  By numerical calculations we guess that if we replace $n$ with $n+x$ in the summands of this series we have the following expansion
\begin{equation}\label{rama-gui2-expan}
\frac{1}{\pi^3}-k \frac{1}{\pi} \frac{x^2}{2!}+j \pi \frac{x^4}{4!}-l \pi^3 \frac{x^6}{6!} + O(x^7),
\end{equation}
with $k=2$, $j=32$ and $l=4112$.

\section*{Acknowledgments} I thank Wadim Zudilin for going over the successive versions of this paper carefully helping with his corrections, important comments and useful suggestions.

\enddocument